\documentclass{amsart}[9pt]

\usepackage{psfrag}

\usepackage{setspace}
\usepackage[totalwidth=14cm,totalheight=24cm]{geometry}
\usepackage[T1]{fontenc}
\usepackage{epsfig}
\usepackage{amsmath,amssymb,amscd,amsthm}
\usepackage[latin1]{inputenc}
\usepackage{colortbl}
\usepackage{ae}
\usepackage{color}

\usepackage{textcomp}
\begin{document}

\setcounter{secnumdepth}{0}
\setcounter{tocdepth}{0}

\newtheorem{definition}{Definition}
\newtheorem{lemma}[definition]{Lemma}
\newtheorem{sublemma}[definition]{Sublemma}
\newtheorem{corollary}[definition]{Corollary}
\newtheorem{proposition}[definition]{Proposition}
\newtheorem{theorem}[definition]{Theorem}

\newtheorem{remark}[definition]{Remark}
\newtheorem{example}[definition]{Example}
\newtheorem{question}[definition]{Question}
\newtheorem{conjecture}[definition]{Conjecture}

\newcommand{\cov}{\mathrm{covol}}
\def \tr{{\mathrm{tr}}}
\def \det{{\mathrm{det}\;}}
\def\co{\colon\tanhinspace}
\def\I{{\mathcal I}}
\def\N{{\mathbb N}}
\def\R{{\mathbb R}}
\def\Z{{\mathbb Z}}
\def\Sph{{\mathbb S}}
\def\Tor{{\mathbb T}}
\def\Disk{{\mathbb D}}
\def\Hess{\mathrm{Hess}}
\def\rad{\mathbf{v}}

\def\H{{\mathbb H}}
\def\RP{{\mathbb R}{\mathrm{P}}}
\def\dS{{\mathrm d}{\mathbb{S}}}
\def\Isom{\mathrm{Isom}}

\def\sh{\mathrm{sinh}\,}
\def\ch{\mathrm{cosh}\,}
\newcommand{\arccosh}{\mathop{\mathrm{arccosh}}\nolimits}
\newcommand{\oh}{\overline{h}}

\renewcommand\b[2]{\langle #1,#2\rangle_-}

\newcommand{\mf}{\mathfrak}
\newcommand{\mb}{\mathbb}
\newcommand{\ol}{\overline}
\newcommand{\la}{\langle}
\newcommand{\ra}{\rangle}
\newcommand{\hess}{\mathrm{Hess}\;}
\newcommand{\grad}{\mathrm{grad}}
\newcommand{\M}{\mathrm{MA}}
\newcommand{\II}{\textsc{I\hspace{-0.05 cm}I}}
\renewcommand{\d}{\mathrm{d}}
\newcommand{\A}{\mathrm{A}}
\renewcommand{\L}{\mathcal{L}}
\newcommand{\FB}[1]{{\color{red}#1}}
\newcommand{\note}[1]{{\color{blue}{\small #1}}}
\newcommand{\ap}{\mathrm{area}^{\scriptscriptstyle +}}
\newcommand{\am}{\mathrm{area}^{\scriptscriptstyle -}}
\renewcommand{\a}{\mathrm{area}}

\renewcommand{\v}{\mathrm{vol}}

\renewcommand{\thefootnote}{\fnsymbol{footnote}}

%%%%%%%%%%%%%%%%%%%%%%%%%%%%%%%%%%%
%%%%%%%% gere espace texte formule
\setlength{\abovedisplayshortskip}{1pt}
\setlength{\belowdisplayshortskip}{3pt}
\setlength{\abovedisplayskip}{3pt}
\setlength{\belowdisplayskip}{3pt}

%%%%%%%%%%%%%%%%%%%%%%%%%%%%%%%%%%
\title{A  short elementary proof of reversed Brunn--Minkowski inequality for coconvex bodies}

\author{Fran\c{c}ois Fillastre}
\address{Universit\'e de Cergy-Pontoise, UMR CNRS 8088, F-95000 Cergy-Pontoise, France}
\email{francois.fillastre@u-cergy.fr}
\date{\today}

\maketitle

%\textbf{Subject Classification (2010)}

\footnotetext{
The author thanks Ivan Izmestiev and Rolf Schneider.}   

\begin{abstract}
The theory of coconvex bodies was formalized by A.~Khovanski{\u\i}
and V.~Timorin in \cite{KT}.  It has fascinating relations with the classical theory of convex bodies, as well as applications to Lorentzian geometry. In a recent preprint \cite{schnei2}, R.~Schneider proved a result that implies a reversed Brunn--Minkowski inequality for coconvex bodies, with description of equality case. 
In this note we show that this latter result is an immediate consequence
of a  more general result, namely that the volume of coconvex bodies is strictly convex. This result itself follows from a classical elementary result about the concavity of the volume of convex bodies inscribed in the same cylinder.
\end{abstract}

%%%%%%%%%%%%%%%%%%%%%%%%%%%%%%%%%%%%%%%%%%%%%%%%%%%%%%%%%%
Let $C$ be a closed convex cone in $\R^n$, with non empty interior, and not containing an entire line. A \emph{$C$-coconvex body} $K$
is a non-empty closed bounded proper subset of $C$ such that $C\setminus K$ is convex. 
The set of $C$-coconvex bodies is stable under positive homotheties. It is also stable 
for the $\oplus$ operation, defined as $K_1\oplus K_2=C\setminus (C\setminus K_1 + C\setminus K_2)$, where $+$ is the Minkowski sum.
The following reversed Brunn--Minkowski theorem is proved in \cite{schnei2} (see \cite{KT} for a partial result). We denote by $V_n$ the volume in $\R^n$.

\begin{theorem}\label{sch}
Let $K_1,K_2$ be $C$-coconvex bodies, and $\lambda \in (0,1)$. Then
$$V_n((1-\lambda)K_1 \oplus \lambda K_2)^{1/n} \leq (1-\lambda)V_n(K_1)^{1/n}+\lambda V_n(K_2)^{1/n}~, $$
and equality holds if and only if $K_1=\alpha K_2$ for some $\alpha >0$.
\end{theorem}

\begin{remark}{\rm 
What is actually proved in \cite{schnei2} in the analogous of Theorem~\ref{sch} for \emph{$C$-coconvex sets} instead of $C$-coconvex bodies: 
the set is not required to be bounded but only to have finite
Lebesgue measure. So the result of 
 \cite{schnei2} requires a more involved proof than the one presented here.
}\end{remark}

Actually, we will see that the following result holds.

\begin{theorem}\label{conv}
The volume is strictly  convex on the set of $C$-coconvex bodies.
More precisely, if $K_1,K_2$ are $C$-coconvex bodies, and $\lambda \in (0,1)$, then
$$V_n((1-\lambda)K_1 \oplus \lambda K_2) \leq (1-\lambda)V_n(K_1)+\lambda V_n(K_2)~. $$
Moreover,  equality holds if and only if $K_1=K_2$.
\end{theorem}

The following elementary lemma, together with the fact that $V_n$ is positively homogeneous of degree $n$ (i.e. $V_n(tA)=t^nV_n(A)$ for $t>0$), shows that Theorem~\ref{conv} implies Theorem~\ref{sch}.

\begin{lemma}
Let $f$ be a positive convex function, positively homogeneous of degree $n$. Then $f^{1/n}$ is convex.

Suppose moreover that  $f$ is strictly convex.  If there is $\lambda\in(0,1)$ such that $f^{1/n}((1-\lambda)x+\lambda y))$ equals $(1-\lambda)f^{1/n}(x)+\lambda f^{1/n}(y)$, then there is $\alpha >0$ with $x=\alpha y$.
\end{lemma}
\begin{proof}
For $\bar \lambda\in[0,1]$ and any $x,y$, we have $f((1-\bar \lambda)\frac{x}{f(x)^{1/n}}+\bar\lambda \frac{y}{f(y)^{1/n}})\leq 1$, and the result follows by taking, for any $\lambda\in (0,1)$, $\bar \lambda=
\lambda f(y)^{1/n} / ((1-\lambda)f(x)^{1/n}+\lambda f(y)^{1/n})$.
\end{proof}

Let us prove Theorem~\ref{conv}.

 Let $H$ be an affine hyperplane of $\R^n$ with the following properties: it has an orthogonal 
direction in the interior of $C$,   
 $K_1,K_2$ and the origin are contained in the same half-space $H^+$ bounded by $H$, and  $H\cap C=B$ is compact. For $\lambda\in [0,1]$, let $K_\lambda=(1-\lambda)K_1\oplus \lambda K_2$, which is also  contained in 
$H^+$, and let
$\operatorname{cap}_H(K_\lambda)=H^+ \cap (C\setminus K_\lambda) $, see Figure~\ref{fig:notations}.

Also, the quantity $V_n(K_\lambda)+V_n(\operatorname{cap}_H(K_\lambda))$ does not depend on $\lambda$, as it is equal to $V_n(C\cap H^+)$.
Hence Theorem~\ref{conv} is equivalent to
$$V_n(\operatorname{cap}_H(K_\lambda)) \geq (1-\lambda)V_n(\operatorname{cap}_H(K_1))+\lambda V_n(\operatorname{cap}_H(K_2))~ $$
for $\lambda\in (0,1)$, with  equality  if and only if $K_1=K_2$.

 This last result itself follows from the following elementary result. Here ``elementary'' means that the most involved technique in its proof is  Fubini theorem  (see Chapter~50 in 
\cite{BF87} or Lemma~3.30 in \cite{bf}).

\begin{lemma}\label{lem:vol concave}
Let
 $A_0$ and $A_1$ be two convex bodies in $\R^{n}$ contained in 
$H^+$, 
such that their orthogonal projection 
onto $H$ is $B$.
Then, for $\lambda\in[0,1]$,
$$V_n((1-\lambda)A_0+\lambda A_1)\geq (1-\lambda) V_n(A_0)+\lambda V_n(A_1)~.$$
Equality holds if and only if either
$A_0=A_1+U$ or $A_1=A_0+U$, where $U$  is some segment whose direction is orthogonal to $H$.
\end{lemma}
In our case, if $K$ is a $C$-coconvex body, then $K\oplus U$ is a $C$-coconvex body if and only if $U=\{0\}$.

\begin{figure}
\begin{center}
\psfrag{C}{$C$}
\psfrag{H}{$H$}
\psfrag{H+}{$H^+$}
\psfrag{K}{$K$}
\psfrag{B}{$B$}
\psfrag{ca}{$\operatorname{cap}_H(K)$}
\includegraphics[width=0.5\linewidth]{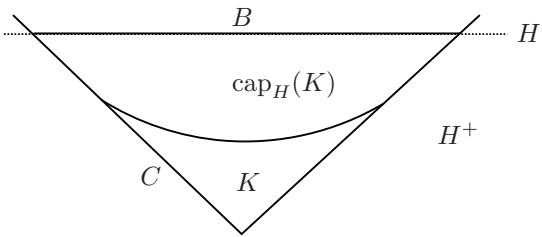}\caption{Notations}\label{fig:notations}
\end{center}
\end{figure}

\begin{remark}{\rm In the classical convex bodies case, the Brunn--Minkowski inequality (saying that the $n$th-root of the volume of convex bodies is concave) follows from the more general result that the volume of convex bodies is log-concave. This is the genuine analogue of our situation, due to the following implications:
\begin{eqnarray*}
 \  f \mbox{ concave} &\Longrightarrow &      f \mbox{ log-concave}  \\
      \  f \mbox{ log convex} &\Longrightarrow &  f \mbox{ convex}~.  
 \end{eqnarray*}
 
 If moreover $f$ is positively homogenous of degree  $n$, we have:     
 \begin{eqnarray*}
  \   f \mbox{ log-concave} &\Longrightarrow &   f^{1/n} \mbox{ concave} \\
   \   f \mbox{ convex } &\Longrightarrow & f^{1/n} \mbox{ convex}~.
\end{eqnarray*}

}\end{remark}

\begin{remark}{\rm
Actually we didn't use the fact that the convex set $C$ is a cone, as the only thing that really matters is the stability of $C$-coconvex bodies under convex combinations. See e.g. \cite{bf} for an application to this more general situation. If $C$ is a cone, the $C$-coconvex bodies are furthermore stable under positive homotheties and $\oplus$, that allows to develop a mixed-volume theory for $C$-coconvex sets, see \cite{Fgafa,KT,schnei2}.

}\end{remark}

\begin{spacing}{0.9}
\begin{footnotesize}
\bibliography{flat}
\bibliographystyle{alpha}
\end{footnotesize}
\end{spacing}

\end{document}